\documentclass[12pt, reqno]{amsart}

\usepackage[margin=1.2in]{geometry}

\usepackage{amssymb,upref}
\usepackage{amsmath, amsthm}
\usepackage{enumerate, graphicx, float}
\usepackage{relsize, mathrsfs, upgreek}

\usepackage{color}
\usepackage{esint}

\usepackage{amsmath,amssymb,amscd,amsxtra,upref}
\usepackage{latexsym}

\theoremstyle{plain}

\newtheorem{theor}{Theorem}
\newtheorem{lemma}[theor]{Lemma}

\newtheorem{coro}[theor]{Corollary}

\newtheorem{remark}[theor]{Remark}
\newtheorem{dfn}[theor]{Definition}

\newtheorem{thm}{Theorem} 

\def\XXint#1#2#3{{\setbox0=\hbox{$#1{#2#3}{\int}$ }
\vcenter{\hbox{$#2#3$ }}\kern-.6\wd0}}

\newcommand{\bee}{\begin{equation}}
\newcommand{\eee}{\end{equation}}
\newcommand{\be}{\begin{equation*}}
\newcommand{\ee}{\end{equation*}}

\newcommand{\eps}{\varepsilon}
\newcommand{\f}{\varphi}

\newcommand{\dcc}{\rm(DC)}

\newcommand{\na}{\mathbb{N}}

\newcommand{\rn}{\mathbb{R}^n}

\newcommand{\normn}[1]{\left\|#1\right\|}

\newcommand{\diam}[1]{{\rm diam}(#1)}

\newcommand{\tin}{T^{-1}}

\newcommand{\N}{\mathcal{N}}

\newcommand{\calo}{\mathcal{O}}

\newcommand{\muf}{\mu_\f}

\newcommand{\fs}{\phi^*}

\newcommand{\dx}{\, dx}
\newcommand{\dy}{\, dy}

\DeclareMathOperator*{\essinf}{ess\,inf\,}

\begin{document}




\thanks{Partially supported by Simons Foundation grant MPS-TSM-00007229}

\address{Diego Maldonado, Kansas State University, Department of Mathematics. 138 Cardwell Hall, Manhattan, KS-66506, USA.} \email{dmaldona@ksu.edu}

\title[]{On a higher-integrability estimate for Laplacians of convex solutions to the Monge-Amp\`ere equation}
\author[Diego Maldonado]{Diego Maldonado}

\thanks{}

\date{\today}

\begin{abstract} A maximum principle for certain strictly convex functions belonging to $W^{2,1}_{\rm loc}(\Omega)\cap C^1(\Omega)$ is established and then applied to obtain a $W^{2,1 +\eps}_{\rm loc}(\Omega)$-estimate.
\end{abstract}

\maketitle

\section{Introduction and main results}

Given an open convex set $\Omega \subset \rn$ and constants $0 < \lambda \leq \Lambda$, the question on establishing  regularity properties for a strictly convex function $\f \in C^1(\Omega)$ satisfying
$\lambda \leq \det D^2 \f \leq \Lambda$ in $\Omega$ in the Aleksandrov sense, that is,
\begin{equation}\label{intro:muf:like:lambdas}
\lambda |F| \leq |\nabla \f(F)| \leq \Lambda |F|\quad \text{for every Borel set } F \subset \Omega,
\end{equation}
lies at the core of the regularity theory for the Monge-Amp\`ere equation. In this respect, we mention L. Caffarelli's groundbreaking $C^{1,\alpha}_{\rm loc}(\Omega)$-estimate where $\alpha=\alpha(\lambda, \Lambda, n) \in (0,1)$ (see \cite{Caffa92, MaPAMS23}) as well as the more recent $W^{2,1+\eps}_{\rm loc}(\Omega)$-estimates due to G. De Philippis, A. Figalli, and O. Savin in \cite{DeFiSa12} and, independently, to T. Schmidt in \cite{Sch}. More precisely, it is proved in \cite[Theorem 1]{DeFiSa12} and \cite[Theorem 1.1]{Sch}  that if $\f \in C^1(\Omega)$ satisfies \eqref{intro:muf:like:lambdas}, then there exist constants $C_0 \geq 1$ and $\eps_0, \kappa \in (0,1)$, depending only on $\lambda, \Lambda,$ and $n$, such that the (weak) reverse-H\"older inequality
\begin{equation}\label{rhphiDeFi}
\left(\fint_{S_\f(x_0, \kappa t)}  \Delta \f(x)^{1+\eps_0} \dx \right)^\frac{1}{1+\eps_0} \leq C_0 \fint_{S_\f(x_0, t)}  \Delta \f(x) \dx
\end{equation}
holds true for every Monge-Amp\`ere section $S_\f(x_0,t) \subset \subset \Omega$ (that is, $\overline{S_\f(x_0,t)} \subset \Omega$). Here, for $x_0 \in \Omega$ and $t > 0$, $S_\f(x_0,t)$ denotes the \emph{Monge-Amp\`ere section} of $\f$ centered at $x_0$ with height $t >0$, defined as the open, bounded, convex set
$$
S_\f(x_0,t) := \{x \in \Omega: \delta_\f(x_0, x) < t \}
$$
where
\begin{equation}\label{def:delta:f}
\delta_\f(x_0, x):= \f(x) - \f(x_0) - \langle \nabla \f(x_0), x - x_0 \rangle \quad \forall x, x_0 \in \Omega.
\end{equation}
In order to prove \eqref{rhphiDeFi}, in \cite[Section 3]{DeFiSa12} and \cite[Section 4]{Sch} the authors first assume that the strictly convex solution to \eqref{intro:muf:like:lambdas} satisfies $\f \in C^2(\Omega)$, then an approximation argument, as in  \cite[Section 5]{Sch}, yields the result for a general $\f \in C^1(\Omega)$.

Now, for a strictly convex function $\f \in C^1(\Omega)\cap W^{2,1}_{\rm loc}(\Omega)$ let us weaken the condition \eqref{intro:muf:like:lambdas} by assuming, instead, the existence of constants $C_\theta >0$ and $\theta \in (0,1)$ such that 
\begin{equation}\label{intro:CF:MA}
\frac{|\nabla \f(F)|}{|\nabla \f(S)|} \leq C_\theta \left(\frac{|F|}{|S|} \right)^{\theta}
\end{equation}
for every Monge-Amp\`ere section $S:=S_\f(x,t) \subset \subset \Omega$ and every Borel set $F \subset S$. Notice that \eqref{intro:muf:like:lambdas} quickly implies \eqref{intro:CF:MA} with $C_\theta:=\Lambda/\lambda$ and $\theta :=1$. Also, \eqref{intro:CF:MA} corresponds to the well-known Coifman-Fefferman condition characterizing Muckenhoupt's  $A_\infty(\rn)$ weights (see for instance \cite[Section 9.3]{GrafakosBook}), but adapted to the Monge-Amp\`ere sections instead of Euclidean balls. In their foundational work on the linearized Monge-Amp\`ere equation, L. Caffarelli  and C. Guti\'errez   \cite{caffaguti1, caffaguti2} introduced and made extensive use of \eqref{intro:CF:MA}, which they denoted as $\muf \in (\mu_\infty)$ and $\muf$ stands for the \emph{Monge-Amp\`ere measure associated to $\f$} defined as $\muf(F) :=|\nabla \f(F)|$ for any Borel set $F \subset \Omega$. 

Back to \eqref{rhphiDeFi}, as pointed out in \cite[Section 3.3]{DeFiSa12} and in \cite[Section 4]{MaCVPDE14}, the reverse-H\"older inequality \eqref{rhphiDeFi} can still be obtained under the condition \eqref{intro:CF:MA} provided that $\f \in C^2(\Omega)$ (or even $\f \in W^{2,n}_{\rm loc}(\Omega)$, as required by the Aleksandrov-Bakelman-Pucci, from now on ABP,  maximum principle). However, an approximation argument, to pass from $\f \in C^2(\Omega)$ to $\f \in W^{2,1}_{\rm loc}(\Omega)\cap C^1(\Omega)$, that quantitatively preserves the condition \eqref{intro:CF:MA} seems to be currently missing from the literature.  

The purpose of this note is to provide a way towards the $W^{2,1+\eps}_{\rm loc}(\Omega)$-estimate \eqref{rhphiDeFi} that circumvents the ABP maximum principle as well as the hypotheses $\f \in W^{2,n}_{\rm loc}(\Omega)$ or $\f \in C^2(\Omega)$. Indeed, we will only assume that $\f \in W^{2,1}_{\rm loc}(\Omega)\cap C^1(\Omega)$ satisfies \eqref{intro:CF:MA} and resort to the fact that, due to \eqref{intro:CF:MA}, $\nabla \f$ preserves sets of measure zero (that is, $\nabla \f$ satisfies Lusin's $(N)$ condition, see Section \ref{sec:Lusin}). In particular, approximating arguments are rendered unnecessary. 

More precisely, inspired by Lemmas 3.2 and 3.3 in \cite{DeFi11} and Lemma 3.1 in \cite{MaCVPDE14}, we prove the following version of the ABP maximum principle without the usual $W^{2,n}_{\rm loc}$-requirement as in \cite[Section 9.1 and Exercise 9.3]{gt} and \cite[Chapter 6, pp.79-87]{ChenWu}.

\begin{lemma}\label{lemma:abp} Fix an open bounded convex set $\mathcal{O} \subset \rn$ such that
$
B(0, 1) \subset \mathcal{O} \subset B(0, n)
$
and a strictly convex function $\phi \in C(\overline{\mathcal{O}}) \cap C^1(\mathcal{O}) \cap W^{2,1}(\mathcal{O})$ such that 
\begin{enumerate}[(i)]
\item $\phi$ vanishes on $\partial \mathcal{O}$,
\item\label{item:def:c} $\phi(y^*) \leq -c$ for some $y^* \in \mathcal{O}$ and $c >0$, and
\item\label{item:nabla:phi:N} $\nabla \phi : \mathcal{O} \to \rn$ satisfies Lusin's $(N)$ condition on $\mathcal{O}$, that is, $|\nabla \phi(E)| =0$ for every $E\subset \calo$ with $|E|=0$.
\end{enumerate}
Then,
\begin{equation}\label{mu:phi:D2phi>c}
\omega_n \left(\frac{c}{4n}\right)^n \leq \mu_\phi \left(\left\{y \in \mathcal{O} : D^2 \phi(y) \geq \frac{c}{n^2} I  \right\}\right),
\end{equation}
where $\omega_n$ denotes the volume of the $n$-dimensional unit ball in $\rn$. 
\end{lemma}

\begin{remark} With Lemma \ref{lemma:abp} in hand, the next result replaces Lemmas 3.2 and 3.3 in \cite{DeFi11} and then implies Lemma 3.4 from \cite{DeFi11} under the assumptions $\f \in W^{2,1}_{\rm loc}(\Omega)\cap C^1(\Omega)$ and \eqref{intro:CF:MA}  only. Then, the reverse-H\"older inequality \eqref{rhphiDeFi} follows from \cite[Lemma 3.4]{DeFi11} as in \cite[Section 3.2]{DeFiSa12} or as in \cite[Section 4]{Sch}. 
\end{remark}

\begin{lemma}\label{lemma:main} Fix an open convex set $\Omega \subset \rn$, a strictly convex function  $\f \in W^{2,1}_{\rm loc}(\Omega)\cap C^1(\Omega)$ satisfying \eqref{intro:CF:MA}. Then, there are constants $0 < \kappa_2  < 1 < C_2 < \infty$ (depending only on $C_\theta$, $\theta$, and $n$) such that for every section $S:=S_\f(x_0, t) \subset \subset \Omega$ there exists a set $E \subset S$ satisfying
\begin{equation}\label{muES}
\kappa_2 \,  \muf(S) \leq \muf(E) 
\end{equation}
as well as
\begin{equation}\label{d2ecc}
\frac{1}{|S|} \int_S \| D^2 \f(x) \| \, dx \leq C_2 \essinf_E \| D^2 \f\|.
\end{equation}
\end{lemma}

\begin{remark}
Recall that, by Aleksandrov's theorem (see for instance \cite[Section 6.4]{EvansGariepy}), the convexity of $\f$ implies the existence of $D^2 \f(x)$ as a nonnegative, symmetric matrix for a.e.\  $x \in \Omega$ and if $\lambda_1(x) \leq \cdots \leq \lambda_n(x)$ denote its eigenvalues, then $\| D^2 \f(x) \| = \lambda_n(x)$ and $\frac{1}{n} \Delta \f(x) \leq \| D^2 \f(x) \| \leq \Delta \f(x)$ for a.e.\  $x \in \Omega$. 
\end{remark}

\section{Preliminaries}\label{sec:preliminaries}

\subsection{Regularity conditions on Borel measures with respect to Monge-Amp\`ere sections} Let $\f \in C^1(\Omega)$ be a strictly convex function. Following \cite{caffaguti1, caffaguti2, gutihuang} we write $\muf \in \dcc_\f$ if there exist constants $C_{DC} \geq 1$ and $\alpha \in (0,1)$ such that for every section $S_\f(x, t) \subset \subset \Omega$ we have
\begin{equation}\label{dc}
\muf(S_\f(x, t)) \leq C_{DC} \, \muf(\alpha S_\f(x,t)), 
\end{equation}
where $\alpha S_\f(x,t)$ is the $\alpha$-contraction of $S_\f(x,t)$ with respect to its center of mass. In \cite[p.426]{caffaguti2}, L. Caffarelli and C. Guti\'errez proved if $\muf$ satisfies \eqref{intro:CF:MA} and if $\delta_2 \in (0, 1)$ is a number chosen so that $\delta_1:=C_\theta \delta_2^\theta < 1$, then $\muf \in \dcc_\f$ with $C:=1/(1-\delta_1)$ and any $\alpha \in (0,1)$ with $1 - \alpha^n < \delta_2$.

\begin{dfn}[Geometric constants] Under the condition $\muf \in \dcc_\f$ all constants   depending only on $C_{DC}$ and $\alpha$ from \eqref{dc} as well as on $n$ will be called \emph{geometric constants}. In particular, if $\muf $ satisfies \eqref{intro:CF:MA}, geometric constants depend only on the constants $C_\theta$ and $\theta$ from \eqref{intro:CF:MA} and $n$.
\end{dfn}

\subsection{The normal mapping of a continuous function and a maximum principle} Let us record here some basic results on the normal mapping of continuous functions and then conclude with a simple maximum principle that follows from \cite[Chapter 6, Section 1]{ChenWu}, or the proof of Lemma 9.2 in \cite[p.221]{gt}, or \cite[Chapter 1, Sections 1--4]{guti}.

Fix an open bounded connected set $\mathcal{O} \subset \rn$ and $h \in C(\mathcal{O})$. \emph{The normal mapping} $\chi : \calo \to \mathcal{P}(\rn)$ of $h$ is the set-valued mapping defined for $y \in \mathcal{O}$ as
$$
\chi(y):=\{p \in \rn: h(x) \leq h(y) + \langle p, x -y \rangle, \: \forall x \in \calo\}
$$
and \emph{the contact set} of $h$ in $\calo$ is defined as
\begin{align*}
\Gamma_h(\mathcal{O}):&=\{y \in \mathcal{O}: \chi(y) \neq \emptyset\}\\
& = \{y \in \mathcal{O}: \text{there exists } p \in \rn \text{ with } h(x) \leq h(y) + \langle p, x-y \rangle \, \forall x \in \mathcal{O}\}.
\end{align*}
Given $y \in \Gamma_h(\mathcal{O})$, the function $x \mapsto h(y) + \langle p, x-y \rangle - h(x)$, for $x \in \calo$, is minimized at $y$. Hence, if $h \in C^1(\calo)$, it follows that $\chi(y) = \{\nabla h(y)\}$ for every $y \in \Gamma_h(\mathcal{O})$ and then 
\begin{equation}\label{Gamma:h:C1}
\Gamma_h(\mathcal{O})=\{y \in \mathcal{O}: h(x) \leq h(y) + \langle \nabla h(y), x-y \rangle \, \forall x \in \mathcal{O}\}.
\end{equation}
On the other hand, by Lemma 1.1 from \cite[p.80]{ChenWu}, if $h \in W^{2,1}_{\rm loc}(\calo) \cap C(\mathcal{O})$, we have
\begin{equation}\label{D2h<0:Gamma}
D^2 h(y) \leq 0 \quad \text{a.e. } y \in \Gamma_h(\mathcal{O}).
\end{equation}
Finally, as a consequence of Lemmas 1.3 and 1.5 from \cite[pp.82--84]{ChenWu}, we obtain the following version of Aleksandrov's maximum principle. 

\begin{thm}\label{thm:max:Aleksandrov} Given an open set $\mathcal{O} \subset \rn$ and $h \in C(\overline{\mathcal{O} })$ satisfying $h \leq 0$ on $\partial \mathcal{O}$ and $h(\xi_0) >0$ for some $\xi_0 \in \mathcal{O}$, we have
\begin{equation}\label{max:Aleksandrov}
h(\xi_0)^n \leq \omega_n^{-1} \diam{\mathcal{O}}^n |\chi(\Gamma^+_h(\mathcal{O}))|,
\end{equation}
where 
$$
\Gamma^+_h(\mathcal{O}):=\Gamma_h(\mathcal{O}) \cap \{y \in \mathcal{O}: h(y) \geq 0\}.
$$
\end{thm}

\begin{remark}\label{rmk:thm:B} Notice that if the assumption $h \in C(\overline{\mathcal{O} })$ in Theorem \ref{thm:max:Aleksandrov} is replaced with $h \in C^1(\overline{\mathcal{O} })$ (so that $\chi(y) = \{\nabla h(y)\}$ for every $y \in \Gamma_h(\mathcal{O})$),  the inequality \eqref{max:Aleksandrov} can be written as
\begin{equation*}
h(\xi_0)^n \leq \omega_n^{-1} \diam{\mathcal{O}}^n |\nabla h(\Gamma^+_h(\mathcal{O}))|.
\end{equation*}
\end{remark}

\subsection{Lusin's $(N)$ property and the area formula}\label{sec:Lusin} Given an open set $\calo \subset \rn$ and a measurable subset $\Gamma \subset \calo$, a measurable mapping $F : \Gamma \to \rn$ is said to satisfy \emph{Lusin's $(N)$ condition on $\Gamma$} if $|F(E)|= 0$ for every set $E \subset \Gamma$ with $|E|=0$.  

We will be interested in mappings in the Sobolev class $W^{1,1}_{\rm loc}(\calo, \rn)$ defined as
$$
W^{1,1}_{\rm loc}(\calo, \rn) :=\{F=(F_1, \ldots, F_n) : \calo \to \rn: F_j \in W^{1,1}_{\rm loc}(\calo), \, j=1, \dots, n\}.
$$
Theorem 1 and Propositions 1 and 2 from \cite[Section 2]{GMS92} (see also \cite[Proposition 1.1]{Maly94}) give

\begin{thm}\label{thm:Lusin} Fix an open set $\calo \subset \rn$ and a measurable subset $\Gamma \subset \calo$. Then, for $F \in W^{1,1}_{\rm loc}(\calo, \rn)$ the following are equivalent:
\begin{enumerate}[(a)]
\item $|F(E)|= 0$ for every set $E \subset \Gamma$ with $|E|=0$ (i.e., Lusin's condition $(N)$  on $\Gamma$).
\item For every measurable subset $\Gamma' \subset \Gamma$ the area formula
\begin{equation}\label{area:form:v1}
\int_{\Gamma'} |\det D F(x)| \dx = \int_{\rn} \N(y, F, \Gamma') \dy
\end{equation}
holds true, where 
$
\N(y, F, \Gamma'):= \# \{x \in \Gamma': F(x) =y\}.
$
\item The change-of-variable formula holds for $F$ on $S$, that is,
$$
\int_\Gamma u(F(x)) |\det D F(x)| \dx =  \int_{\rn} u(y) \N(y, F, \Gamma) \dy,
$$
for every nonnegative Borel measurable function $u$ defined on $\rn$. 
\end{enumerate}
\end{thm}

Let us now record the following consequence of Theorem \ref{thm:Lusin}.

\begin{coro}\label{coro:thm:lusin} Fix an open set $\calo \subset \rn$ and a subset $\Gamma \subset \calo$. Let $h \in W^{2,1}(\calo)$ such that $\nabla h: \calo \to \rn$ satisfies Lusin's $(N)$ condition on $\Gamma$, then
$$
|\nabla h(\Gamma)| \leq \int_\Gamma |\det D^2 h(x)| \dx.
$$
\end{coro}

\begin{proof}  By using Theorem \ref{thm:Lusin} with $F:=\nabla h \in W^{1,1}_{\rm loc}(\calo, \rn)$ and by noticing that 
$$
\mathbf{1}_{\nabla h(\Gamma)}(y) \leq \N(y, \nabla h, \Gamma) \quad \forall y \in \rn
$$
and that $D F = D^2 h$, the equality \eqref{area:form:v1} gives 
$$
|\nabla h(\Gamma)|  = \int_{\rn} \mathbf{1}_{\nabla h(\Gamma)}(y) \dy \leq \int_{\rn} \N(y, \nabla h, \Gamma) \dy= \int_{\Gamma} |\det D^2 h(x)| \dx. 
$$
\end{proof}

\begin{remark} If $\f \in C^1(\Omega)$ is a strictly convex function verifying \eqref{intro:CF:MA}, it follows that $\nabla \f : \Omega \to \rn$ satisfies Lusin's $(N)$ condition on $\Omega$. In particular, the Monge-Amp\`ere measure associated to $\f$ has $\det D^2 \f$ as its density, that is, by Theorem \ref{thm:Lusin}, we have
$$
\muf(F) = \int_F \det D^2 \f(x) \dx \quad \text{for every Borel set } F \subset \Omega,
$$
because $\N(y, \nabla \f, F):= \# \{x \in F: \nabla \f(x) =y\} = \mathbf{1}_{\nabla \f(F)}(y)$ for every $y \in \rn$, since the strict convexity of $\f$ makes $\nabla \f$ a one-to-one mapping. 
\end{remark}

\subsection{The Caffarelli-Guti\'errez normalization technique}\label{sec:normalization} Here we revisit the Caffarelli-Guti\'errez normalization technique from \cite[Section 1]{caffaguti2} and \cite[Section 3.2]{guti}. 

Given $x_C \in \rn$ and a symmetric, positive-definite $n \times n$ matrix $\Sigma$, set  $A:=\Sigma^{-1/2}$ and define the open, bounded, convex 
$$
E(x_C, A):= \{x \in \rn: |A^{-1}(x - x_C)| < 1\} = \{x \in \rn: \langle \Sigma (x - x_C), x - x_C \rangle < 1\}.
$$
Given a number $a > 0$, the dilation $a E(x_C, A)$ is defined as the ellipsoid $E(x_C, a A)$.

Let $\f \in C^1(\Omega)$ be a strictly convex function and fix a section $S:=S_\f(x_0, t) \subset \subset \Omega$. By F. John's lemma (see for instance \cite[p.166]{Figalli}), there exists an ellipsoid $E(x_C, A)$ (which is the unique ellipsoid of maximal volume contained in $S$) such that
\begin{equation*}
E(x_C, A) \subset S \subset n E(x_C, A).
\end{equation*}
In particular, the affine transformation $Tx := A^{-1}x - A^{-1} x_C$ satisfies $T(E(x_C, A)) = B(0,1)$ as well as
\begin{equation}\label{norm:T(S):B}
B(0, 1) \subset T(S) \subset B(0, n),
\end{equation}
and we say that \emph{$T$ normalizes $S$} and write $S^*:=T(S)$. Set $M:=A^{-1}$. We will use $\|M\|$ to indicate the operator norm of $M$, that is, $\|M\| = 1/\rho_1$, where $\rho_1>0$ is the smallest eigenvalue of $A$. As mentioned, by abuse of notation we will also write $\|T\|$ for $\|M\|$. Next, define $\lambda_0 >0$ as
\begin{equation}\label{def:lambda0}
\lambda_0 := \left(\frac{\muf(S)}{\det(M)} \right)^\frac{1}{n}
\end{equation}
and let $\f^*$ be the strictly convex, continuously differentiable function defined on $T(\Omega)$ as
$$
\f^*(y):= \frac{1}{\lambda_0} \f(\tin y) \quad \forall y \in T(\Omega),
$$
so that
\begin{equation}\label{grad2f:grad2f*}
\nabla \f^*(y) = \frac{1}{\lambda_0} (M^{-1})^t \nabla \f(\tin y)
\end{equation}
as well as
\begin{equation}\label{D2f:D2f*}
D^2 \f^*(y) = \frac{1}{\lambda_0} (M^{-1})^t D^2\f(\tin y)M^{-1}.
\end{equation}
Notice that given a Borel set $\mathcal{E} \subset \rn$ the definition of Monge-Amp\`ere measure and \eqref{grad2f:grad2f*} imply
\begin{equation}\label{muf*:muf}
\mu_{\f^*}(\mathcal{E})=|\nabla \f^*(\mathcal{E}) | = \frac{1}{\lambda_0^n} \det(M^{-1}) |\nabla \f(\tin(\mathcal{E}))| = \frac{1}{\lambda_0^n \det(M)}  \mu_{\f}(\tin(\mathcal{E})).
\end{equation}
In particular, by taking $\mathcal{E}=S^*$ in \eqref{muf*:muf} the choice of $\lambda_0$ in \eqref{def:lambda0} gives
\begin{equation}\label{muS1}
\mu_{\f^*}(S^*) =1.
\end{equation}
Also, if $y:=Tx$ for a.e.\ $x \in \Omega$ \eqref{D2f:D2f*} yields
\begin{equation}\label{D2f<D2fsT2}
\|D^2\f(x)\| = \lambda_0 \|M^t D^2 \fs(y) M\| \leq \lambda_0 \|T\|^2 \|D^2 \fs(y)\|.
\end{equation}
We also have
$$
\delta_{\f^*}(Tx, Tx') = \frac{1}{\lambda_0} \delta_\f(x, x') \quad \forall x, x' \in \Omega
$$
so that
\begin{equation}\label{S:S*}
S_{\f^*}(Tx, t'/\lambda_0) = T(S_\f(x, t')) \quad \forall x \in \Omega,   t' > 0,
\end{equation}
and the identity \eqref{S:S*} gives 
$$
S^*_\sigma:= S_\fs\left(y_0, \frac{\sigma t}{\lambda_0}\right) = T(S(x_0, \sigma t)) \quad \forall \sigma \in (0, 1].
$$
By \cite[Theorem 8]{forzamaldona} if $\muf \in \dcc_\f$ then there is a geometric constant $\kappa_1 \in (0,1)$ such that
\begin{equation}\label{def:kappa1}
\kappa_1 \leq \frac{t}{\lambda_0} \leq \kappa_1^{-1}.
\end{equation}
If, in addition, $\f \in W^{2,1}_{\rm loc}(\Omega)$, from the first few lines of the proof of Theorem 2 in \cite{caffaguti2}, or by Lemma 3.2.1 in \cite{guti}, or by Lemma 3.2 in \cite{DeFi11}, or by Lemma 2.3 in \cite{MaConvexPDE}, there exists a geometric constant $K_1 > 0$  such that
\begin{equation}\label{bounddeltafs}
\int_{S^*} \Delta \fs(y) \, dy \leq K_1.
\end{equation}

\section{Proofs of Lemmas \ref{lemma:abp} and  \ref{lemma:main}}\label{sec:proof:thm:main}

\subsection*{Proof of Lemma \ref{lemma:abp}.} For $c>0$ as in \eqref{item:def:c}, introduce the paraboloid
$$
q(y):= \frac{c}{2 } \left(\frac{|y|^2}{n^2} - 1\right)  \quad \forall y \in \rn,
$$
and the function $h \in C(\overline{\mathcal{O}}) \cap C^1(\mathcal{O}) \cap W^{2,1}(\mathcal{O})$ as
\begin{equation}\label{def:h(y)}
h(y):= q(y)  - \phi(y) \quad \forall y \in \overline{\mathcal{O}}.
\end{equation}
In particular, the assumption \eqref{item:def:c} implies
$$
h(y^*)=q(y^*) - \phi(y^*) \geq - \frac{c}{2 } + c = \frac{c}{2} > 0
$$
Now, since $\phi$ vanishes on $\partial \calo$ and $q < 0$ on $B(0,n) \supset \calo$ we get $h \leq 0$ on $\partial \calo$ and then the maximum principle from Theorem \ref{thm:max:Aleksandrov}  and Remark \ref{rmk:thm:B} (used with $\xi_0 := y^*$) yield
\begin{equation*}
h(y^*)^n \leq \omega_n^{-1} \diam{\mathcal{O}}^n |\nabla h(\Gamma^+_h(\mathcal{O}))|
\end{equation*}
which, along with the fact that $\calo \subset B(0, n)$, gives
\begin{equation}\label{after:max:Aleksandrov:h:C1}
\left(\frac{c}{2}\right)^n \leq \omega_n^{-1} (2n)^n |\nabla h(\Gamma^+_h(\mathcal{O}))|.
\end{equation}
Next, since $h \in C^1(\mathcal{O})$, by \eqref{Gamma:h:C1} the contact set $\Gamma_h(\mathcal{O})$ is given by 
\begin{equation*}
\Gamma_h(\mathcal{O})=\{y \in \mathcal{O}: h(x) \leq h(y) + \langle \nabla h(y), x-y \rangle \, \forall x \in \mathcal{O}\}
\end{equation*}
so that, from the definition of $h$ in \eqref{def:h(y)}, $y \in \Gamma_h(\mathcal{O})$ if and only if
$$
q(x) - q(y) - \langle \nabla q(y), x- y \rangle \leq \phi(x) - \phi(y) - \langle \nabla \phi(y), x- y \rangle \quad \forall x \in \mathcal{O},
$$
which means precisely $\delta_q(y,x) \leq \delta_\phi(y,x)$ for every $x \in \mathcal{O}$. A quick computation shows that
$$
\delta_q(y,x) = \frac{c}{2n^2} |x-y|^2 \quad \forall x, y \in \rn,
$$
and then we conclude that $y \in \Gamma_h(\mathcal{O})$ if and only if
\begin{equation}\label{dq<dphi}
 \frac{c}{2n^2} |x-y|^2\leq \delta_\phi(y,x) \quad \forall x \in \mathcal{O}.
\end{equation}
By combining \eqref{dq<dphi} with the fact that the definition of $\delta_\f$ from \eqref{def:delta:f} gives
$$
\delta_\phi(y,x) \leq \delta_\phi(y,x)  + \delta_\phi(y,x) = \langle \nabla \f(x) - \nabla \f(y), x - y \rangle  \quad \forall x, y \in \mathcal{O},
$$
we obtain that 
\begin{equation}\label{dxy<nablafxy}
 \frac{c}{2n^2} |x-y|\leq |\nabla \phi(x) - \nabla \phi(y)| \quad \forall x \in \mathcal{O},  \forall y \in \Gamma_h(\mathcal{O}).
\end{equation}
The novelty of this proof consists in showing the fact that $\nabla h : \mathcal{O} \to \rn$ satisfies Lusin's $(N)$ condition on $\Gamma_h(\mathcal{O})$. That is, given a subset $Z \subset \Gamma_h(\mathcal{O})$ with $|Z| = 0$ we will show that for every $\eps > 0$ there exists a family of Euclidean balls $\{B_j\}_{j \in \na}$ with 
\begin{equation}\label{h(Z):measure:0}
\nabla h(Z) \subset \bigcup\limits_{j \in \na} B_j \quad \text{and} \quad \sum\limits_{j \in \na} |B_j| < \eps. 
\end{equation}
Hypothesis \eqref{item:nabla:phi:N}, along with the fact that $|Z| = 0$, implies  $|\nabla \phi(Z)| =0$. Therefore, given $\eps > 0$ there exists a family of Euclidean balls $\{B(y_j, \rho_j)\}_{j \in \na}$ with 

\begin{equation}\label{phi(Z):UBj}
\nabla \phi(Z) \subset \bigcup\limits_{j \in \na} B(y_j, \rho_j)
\end{equation}
and
\begin{equation}\label{sum:|Bj|:eps:6m}
\sum\limits_{j \in \na} |B(y_j, \rho_j)| < \frac{\eps}{6^n}.
\end{equation}
Notice that given $j \in \na$ we can assume that $\nabla \phi(Z) \cap B(y_j, \rho_j) \neq \emptyset$; otherwise the ball $B(y_j, \rho_j)$ does not participate in the union in \eqref{phi(Z):UBj}. Hence, we must have $\nabla \phi(z_j) \in B(y_j, \rho_j)$ for some $z_j \in Z$, which along with the inclusion \eqref{phi(Z):UBj} yields 
$$
\nabla \phi(Z) \subset \bigcup\limits_{j \in \na} B(\nabla \phi(z_j), 2 \rho_j).
$$
Now,  for $j \in \na$ set $w_j:=\frac{c}{n^2} z_j - \nabla \phi(z_j)$ and let us see that
\begin{equation}\label{nabla:h:(Z):wj}
\nabla h(Z) \subset \bigcup\limits_{j \in \na} B(w_j, 6 \rho_j).
\end{equation}
Indeed, for $z \in Z$ there exists $j_z \in \na$ such that $\nabla \f(z) \in B(\nabla \phi(z_{j_z}), 2 \rho_{j_z})$ and then
\begin{align*}
|\nabla h(z)  - w_{j_z}|& = |\frac{c}{n^2} z - \nabla \phi(z) - w_{j_z}| \leq \frac{c}{n^2} |z - z_{j_z}|+ |\nabla \phi(z) - \nabla \phi(z_{j_z})|\\
& \leq 3 |\nabla \phi(z) - \nabla \phi(z_{j_z})| < 6 \rho_{j_z},
\end{align*}
 where for the second inequality we used \eqref{dxy<nablafxy} (with $y:=z_{j_z}$ and $x:=z$). Then, the inclusion \eqref{nabla:h:(Z):wj} is proved and by setting $B_j:=B(w_j, 6 \rho_j)$ for $j \in \na$, \eqref{h(Z):measure:0} follows from \eqref{nabla:h:(Z):wj} and \eqref{sum:|Bj|:eps:6m}. Consequently, $\nabla h : \mathcal{O} \to \rn$ possesses Lusin's $(N)$ property on $\Gamma_h(\mathcal{O})$, as claimed. Therefore, by Corollary \ref{coro:thm:lusin} applied with $\Gamma := \Gamma_h(\mathcal{O})$ we obtain
\begin{equation}\label{nabla:h:<:det:D2h:Gamma}
|\nabla h(\Gamma_h(\mathcal{O}))| \leq \int_{\Gamma_h(\mathcal{O})} |\det D^2 h(y)| \dy.
\end{equation}
Now, by \eqref{D2h<0:Gamma} we have $D^2 h(y) \leq 0$ for a.e.\  $y \in \Gamma_h(\mathcal{O})$. If $\lambda_1(y) \leq \cdots \lambda_n(y)$ denote the eigenvalues of $D^2\phi(y)$, then the eigenvalues of $D^2 h(y) = \frac{c}{n^2}I - D^2\phi(y)$ are of the form
$\frac{c}{n^2} - \lambda_j(y)$, $j=1, \ldots, n$, and the condition  $D^2 h(y) \leq 0$  means 
\begin{equation}\label{D2phi>cn2:in:Gamma:h}
\frac{c}{n^2}I \leq  D^2\phi(y) \quad \text{a.e. } y \in \Gamma_h(\mathcal{O}),
\end{equation}
that is, $\frac{c}{n^2} \leq \lambda_j(y)$ for  $j=1, \ldots, n.$ Thus, for a.e.\  $y \in \Gamma_h(\mathcal{O})$ we have
\begin{align*}
|\det D^2 h(y)| = \prod\limits_{j=1}^n \left|\frac{c}{n^2} - \lambda_j(y)\right| = \prod\limits_{j=1}^n \left(\lambda_j(y) - \frac{c}{n^2}\right) \leq \prod\limits_{j=1}^n \lambda_j(y) = \det D^2 \phi(y),
\end{align*}
which together with \eqref{nabla:h:<:det:D2h:Gamma} implies
\begin{equation}\label{nabla:h:<:det:D2phi:Gamma}
|\nabla h(\Gamma_h(\mathcal{O}))| \leq \int_{\Gamma_h(\mathcal{O})} |\det D^2 \phi(y)| \dy.
\end{equation}
Next, the strict convexity of $\phi$ makes $\nabla \phi : \calo \to \rn$ a 1-to-1 mapping; in particular, we have
$$
\mathbf{1}_{\nabla \phi(\Gamma_h(\mathcal{O}))}(y) = \N(y, \nabla \phi, \Gamma_h(\mathcal{O})) \quad \forall y \in \rn
$$
which, along with the fact that $\nabla \phi$ satisfies Lusin's $(N)$ condition in $\calo$ and with \eqref{area:form:v1} from Theorem \ref{thm:Lusin} (used with $F=\nabla \phi$, $\Gamma=\calo$, and $\Gamma'=\Gamma_h(\mathcal{O})$), yields
\begin{equation}\label{area:form:D2phi:Gamma}
\int_{\Gamma_h(\mathcal{O})} |\det D^2 \phi(y)| \dy = \mu_\phi(\Gamma_h(\mathcal{O})).
\end{equation}
Finally, \eqref{mu:phi:D2phi>c} follows from \eqref{after:max:Aleksandrov:h:C1}, \eqref{nabla:h:<:det:D2h:Gamma}, \eqref{nabla:h:<:det:D2phi:Gamma}, \eqref{area:form:D2phi:Gamma}, and \eqref{D2phi>cn2:in:Gamma:h}. \qed

\subsection*{Proof of Lemma \ref{lemma:main}} Let us keep with the notation from Section \ref{sec:normalization} and set $S^*:=T(S)$. By \eqref{D2f<D2fsT2}  we can write
\begin{align}\nonumber
\frac{1}{|S|}\int_S \normn{D^2 \f(x)} \, dx \leq \frac{\lambda_0 \|T\|^{2}}{|S|} \int_S \Delta \fs(Tx)\, dx & \leq \frac{\lambda_0 \|T\|^{2}}{|S| |T|} \int_{S^*} \Delta \fs(y)\, dy\\\label{upperbound}
&\leq  \frac{K_1 \lambda_0 \|T\|^{2}}{\omega_n},
\end{align}
where for the last inequality we used that $|S| |T| \geq |B(0,1)| =:\omega_n$ (which follows from the first inclusion in \eqref{norm:T(S):B}) and \eqref{bounddeltafs}. Next, set $y_0:=T(x_0)$ and for $y \in T(\Omega)$ define
$$
\phi(y):= \fs(y) - \fs(y_0) - \langle \nabla \fs(y_0), y - y_0 \rangle  - \frac{t}{\lambda_0}
$$
which makes $\nabla \phi(y) = \nabla \fs(y) - \nabla \fs(y_0)$ for every $y \in T(\Omega)$ and consequently $\mu_\phi = \mu_{\fs}.$ Also, since $S^*:=T(S)$, the identity \eqref{S:S*} yields $S^*= S_{\fs}(y_0, t/\lambda_0)$ and then
$$
\phi =0 \text { on } \partial S^* \quad \text{ and } \quad \min\limits_{S^*} \phi =\phi(y_0)= -\frac{t}{\lambda_0} \leq - \kappa_1.
$$
with $\kappa_1 \in (0,1)$ being the geometric constant from \eqref{def:kappa1}. 

Then, by Lemma \ref{lemma:abp} applied to $\phi$ (with $\calo=S^*$, $y^*=y_0$, and $c=\kappa_1$) we obtain
\begin{equation}\label{muE>kappa2}
\kappa_2 \leq \mu_{\phi}(E^*),
\end{equation}
where $\kappa_2 :=\omega_n \left(\frac{\kappa_1}{4n}\right)^n >0$ (a geometric constant) and
\begin{equation}\label{def:E*}
E^*:= \left\{y \in S^* : D^2 \phi(y) \geq \frac{\kappa_1}{n^2} I\right\}.
\end{equation}
Letting $E:=T^{-1}(E^*) \subset S$, the inequality \eqref{muE>kappa2} and the identities \eqref{muS1} and \eqref{muf*:muf}  yield
\begin{equation*}
\kappa_2 \leq \mu_{\phi}(E^*) =  \mu_{\phi^*}(E^*) = \frac{\mu_{\phi^*}(E^*)}{\mu_{\phi^*}(S^*)} =  \frac{\mu_{\f}(E)}{\mu_{\f}(S)}
\end{equation*}
and \eqref{muES} follows.  Next, by \eqref{def:E*}, $y \in E^*$ if and only if
$
D^2 \phi(y) =D^2 \phi^*(y) \geq \frac{\kappa_1}{n^2}I, 
$
which means (think $x = T^{-1}y$ and recall \eqref{D2f:D2f*}) 
$$
D^2 \f(x) = \lambda_0 M^t D^2 \fs(y) M \geq \frac{\kappa_1 \lambda_0}{n^2} M^tM \quad \forall x \in E.
$$
Consequently, 
\begin{equation*}
\|D^2 \f(x)\| \geq \frac{\kappa_1 \lambda_0}{n^2} \|T^tT\| = \frac{\kappa_1 \lambda_0}{n^2} \|T\|^2 \quad \forall x \in E
\end{equation*}
which together with \eqref{upperbound} yields \eqref{d2ecc} with $C_2:= n^2K_1/(\omega_n \kappa_1).$ \qed

\section*{Acknowledgements}  This work has been partially supported by the Simons Foundation's grant MPS-TSM-00007229.

\section*{} Conflict of Interest Statement: The author declares no conflict of interest.

Data Availability Statement: No data were generated or analyzed in this study.

\end{document}